\documentclass[11pt]{article}
\usepackage{titlesec}
\usepackage{titlesec}
\titleformat{\section}[block]
  {\fontsize{12}{15}\filcenter}
  {\thesection}
  {1em}
  {\MakeUppercase}
\titleformat{\subsection}[runin]
  {\fontsize{12}{15}\bfseries}
  {\thesubsection}
  {1em}
  {}

\usepackage{amsmath,amsthm,amsfonts,amssymb,mathrsfs,epsfig,bm}
\usepackage[usenames]{color}
\definecolor{Maroon}{rgb}{0.69, 0.19, 0.38}
\definecolor{Fuchsia}{rgb}{1.0, 0.0, 1.0}
\usepackage[colorlinks=true]{hyperref}
\newtheorem{theorem}{Theorem}[section]

\newtheorem{lemma}[theorem]{Lemma}
\newtheorem{corollary}[theorem]{Corollary}
\newtheorem{definition}[theorem]{Definition}
\theoremstyle{remark} \newtheorem{remark}{\bf\em Remark}
\theoremstyle{remark} 
\theoremstyle{remark} 
\def\N{\mathbb{N}}

\def\R{\mathbb{R}}

\def\P{\mathbb{P}}
\def\E{\mathbb{E}}
\def\FF{\mathscr{F}}
\def\GG{\mathscr{G}}

\def\BB{\mathscr{B}}
\def\PP{\mathscr{P}}
\def\SS{\mathscr{S}}

\def\LL{\mathcal{L}}

\renewcommand{\phi}{\varphi}
\renewcommand{\epsilon}{\varepsilon}
\newcommand{\1}{{\text{\Large $\mathfrak 1$}}}
\newcommand{\comp}{\raisebox{0.1ex}{\scriptsize $\circ$}}

\newcommand{\var}{\operatorname{var}}
\newcommand{\cov}{\operatorname{cov}}

\newcommand{\eqdist}{\stackrel{\text{\rm (d)}}{=}}

\definecolor{mygray}{gray}{0.6}
\definecolor{deeppink}{RGB}{255,20,147}
\definecolor{mygreen}{rgb}{0.05, 0.576, 0.03}
\definecolor{myred}{rgb}{0.768, 0.09, 0.09}

\long\def\symbolfootnote[#1]#2{\begingroup
\def\thefootnote{\fnsymbol{footnote}}\footnote[#1]{#2}\endgroup}
\newcommand{\keywords}[1]{ \noindent {\footnotesize
             {\small \em Keywords and phrases.} {\sc #1} } }
\newcommand{\ams}[2]{  \noindent {\footnotesize
             {\small \em AMS {\rm 2010} subject classification.
             {\rm Primary {\sc #1}; secondary {\sc #2}} } } }
\newcounter{para}

\def\NN{\mathscr N}

\def\L{\mathcal L}
\usepackage{authblk}
\def\S{\mathfrak S}
\def\Phieq{\bm[\Phi^*\bm]}
\definecolor{MidnightBlue}{rgb}{0.1, 0.1, 0.44}

\def\T{\mathcal T}
\def\D{\mathbb D}

\theoremstyle{plain} \newtheorem*{theorem*}{Theorem}
\allowdisplaybreaks
\title{\bf On the extendibility of finitely exchangeable probability measures} 
\author[1]{Takis Konstantopoulos$^*$}
\author[2]{Linglong Yuan$^{*\bullet}$}
\affil[1]{Department of Mathematics, Uppsala University}
\affil[2]{Institut f\"ur Mathematik,
Johannes-Gutenberg-Universit\"at Mainz}

\begin{document}

\maketitle

\begin{abstract}
A length-$n$ random sequence $X_1,\ldots,X_n$ 
in a space $S$ is finitely exchangeable if its distribution is
invariant under all $n!$ permutations of coordinates.
Given $N > n$, we study the extendibility problem: when is it the
case that there is a length-$N$ exchangeable random sequence $Y_1,\ldots, Y_N$
so that $(Y_1,\ldots,Y_n)$ has the same distribution as $(X_1,\ldots,X_n)$?
In this paper, we give a necessary and sufficient condition so that,
for given $n$ and $N$, the extendibility problem admits a solution.
This is done by employing functional-analytic and measure-theoretic arguments
that take into account the symmetry.
We also address the problem of infinite extendibility. 
Our results are valid when $X_1$ has a regular distribution in
a locally compact Hausdorff space $S$.
We also revisit the problem of representation of the distribution of
a finitely exchangeable sequence.

\symbolfootnote[0]{$\!\!\!^*$Supported by Swedish Research
Council grant 2013-4688}
\symbolfootnote[0]{$\!\!\!^\bullet$Supported by German DFG priority programme SPP 1590 {\it Probabilistic
structures in evolution}}

\ams{60G09, 	
28C05   	
}
{46B99,   	
28A35   	
62F15,   	
28C15   	
}
\\[1mm]
\end{abstract}


\section{Introduction and motivation}
Exchangeability is one of the most important topics in probability theory
with a wide range of applications.
Most of the literature is concerned with exchangeability
for an infinite sequence $X=(X_1,X_2,\ldots)$ of random variables,
taking values in some common space $S$,
in the sense that the distribution of the sequence does not
change when we permute finitely many of the variables. David Aldous'
survey \cite{Aldous85} of the topic gives a very good overview.
The basic theorem in the area is de Finetti's theorem 
stating that, under suitable assumptions on $S$,
an exchangeable sequence is a mixture of i.i.d.\
random variables \cite{DeFi37, Dynkin53, HewSav55, Ka02};
that is, there is a probability measure $\mu$ on the space
$\PP(S)$ of probability measures on $S$ such that
\begin{equation}
\label{DeFin}
\P(X \in \cdot) = \int_{\PP(S)} \pi^\infty(\cdot)\, \mu(d\pi),
\end{equation}
and, moreover, the so-called mixing measure $\mu$ is unique.
The most usual condition on $S$  is that it be 
a Borel space (a measure space that is measure-isomorphic
to a Borel subset of the real line). In this case,
the regular conditional distribution $\eta$ of $X_1$ 
given the invariant $\sigma$-algebra of $(X_1,X_2,\ldots)$ exists,
and the mixing measure $\mu$ is simply 
the law of the random measure $\eta$. The idea of this is due
to Ryll-Nardzewski \cite{RN} and a modern proof of it
can be found in Kallenberg \cite{Ka02}.
Another condition on $S$ is that it be a locally compact
 Hausdorff space
equipped with the $\sigma$-algebra of Baire sets; see
Hewitt and Savage \cite{HewSav55}.
However, \eqref{DeFin} fails for a general space $S$;
see Dubins and Freedman \cite{DubFree79} for a classical counterexample.

Throughout the paper, when $(S,\SS)$ is a measurable space,
$S^n$ is equipped with the product $\sigma$-algebra $\SS^n$, and $S^\N$,
the set of sequences in $S$, is equipped with the $\sigma$-algebra $\SS^{\N}$
generated by cylinder sets.
Also, the set $\PP(S)$ of probability measures on $(S,\SS)$
is equipped with the smallest $\sigma$-algebra that
makes the functions $\PP(S) \ni \pi \mapsto \pi(B) \in \R$, 
where $B$ ranges over $\SS$,
measurable.

Our paper is concerned with 
finitely exchangeable sequences $(X_1,\ldots, X_n)$ 
of a fairly general space $S$. We say that $(X_1,\ldots, X_n)$
is $n$-exchangeable, or, simply, exchangeable, if its law is invariant
under any of the $n!$ permutations of the variables. 
Finitely exchangeable sequences appear naturally in biology models
e.g., in the exchangeable external branch lengths in coalescent
processes \cite{Yuan13, Yuan14, Yuan15, Ker11, Ker14},
as well as in statistical physics models \cite{chalmers07}.

The following issues are
well-known. First, finitely exchangeable sequences need not be
mixtures of i.i.d.\ random variables. Second, finitely exchangeable sequences
of length $n$ may not be extendible to longer 
(finite or infinite) exchangeable sequences.

Regarding the first issue, there is the following, at first surprising,
result: 
\begin{theorem} [finite exchangeability representation result]
\label{rrr}
Let $(S,\SS)$ be an arbitrary measurable space and
let $(X_1,\ldots,X_n)$ be an $n$-exchangeable
random element
of $(S^n,\SS^n)$.
Then there is a {\em signed} measure
$\nu$, with finite total variation,
on the space $\PP(S)$ of probability measures on $S$, such that
\begin{equation}
\label{fre}
\P((X_1, \ldots, X_n) \in A)
= \int_{\PP(S)} \pi^n(A)\, \nu(d \pi),
\quad \forall A \in \SS^n.
\end{equation}
\end{theorem}
This result was first proved by Jaynes \cite{Jay86} for the case
where $S$ is a 2-element set. See also the paper
of Diaconis \cite{Diaconis77} for a clear discussion of the geometry
behind this formula. The general case, i.e., for arbitrary measurable
space $(S,\SS)$, was considered by Kerns and Sz\'ekely \cite{KerSze06}.
In \cite{JKY} we gave a short, complete, and general proof of the formula
(also clarifying/correcting some subtle points of \cite{KerSze06})
and established some notation which is also used
in this current paper.
We express the above result by saying that the law of an
$n$-exchangeable random vector is a  (signed) mixture
of product measures.
Moreover, $\nu$ is not necessarily unique and, typically, it is not.
The signed measure $\nu$ is referred to as a directing signed measure.
By ``signed'' we of course mean ``not necessarily nonnegative''.

Consider now the second issue, that of extendibility. 
Let $(S, \SS)$ be a measurable space and $(X_1,\ldots,X_n)$
an $n$-exchangeable sequence of random elements of $S$.
\\
{\em (a) Finite extendibility.}
For integer $N > n$, we say that $(X_1,\ldots,X_n)$ is $N$-extendible
if there is an $N$-exchangeable sequence $(Y_1,\ldots,Y_N)$ of
random elements of $S$ such that
$(X_1,\ldots,X_n) \eqdist (Y_1,\ldots, Y_n)$, where $\eqdist$
denotes equality in distribution.
\\
{\em (b) Infinite extendibility.}
We say that $(X_1,\ldots,X_n)$ is infinitely-extendible, if there is an 
infinite sequence $(Y_1,Y_2,\ldots)$ that is exchangeable and
$(X_1,\ldots,X_n) \eqdist (Y_1,\ldots, Y_n)$.
\\
Alternatively, we say that an exchangeable probability measure $P_n$ on $S^n$,
invariant under all $n!$ coordinate permutations,
is $N$-extendible if there is an exchangeable probability measure $P_N$ on $S^N$
such that $P_n(A) = P_N(A \times S^{N-n})$, for all $A \in \SS^n$.
Similarly for infinite extendibility.

It is not difficult to see that
a finitely exchangeable sequence may not be extendible. 
For instance,
let $P_n$ be the probability measure corresponding to sampling 
without replacement from an urn with $n \ge 2$ different items. 
Specifically, let $S=\{0,1\}$ and $P_n$ the uniform probability
measure on the set $(x_1,\ldots,x_n) \in S^n$ such that
the number of $x_i$ that are equal to $1$ is fixed and equal to $k$, say.
(The value of $P_n$ at a single $(x_1,\ldots,x_n)$ equals 
$1/\binom{n}{k}$.)
Then it is easy to see that there is no $N>n$ 
for which $P_n$ is $N$-extendible.

One of the goals of this paper is to give necessary and 
sufficient conditions for $N$-extendibility.
A trivial case is that of mixture of product measures, i.e., if 
$P_n(A) = \int_{\PP(S)} \pi^n(A)\, \nu(d\pi)$ for some
{\em probability measure} $\nu$ on $\PP(S)$ then, clearly,
$P_n$ is $N$-extendible. 
But this case is far from being necessary.

The extendibility problem has attracted some attention in the literature. 
For a finite set $S$
the finite extendibility problem reduces to 
the problem of determining whether a point is located in a convex set 
in a multidimensional real vector space. 
This geometric point of view was initiated 
by de Finetti \cite{DeFi69} and further pursued 
by Crisma \cite{crisma71,crisma82}, Spizzichino \cite{spi82},
and Wood \cite{wood92}. The complexity of the problem increases when 
$S$ is an infinite set.
In this case, there is no general method characterizing finite extendibility.
An important contribution is made by Diaconis \cite{Diaconis77} 
and Diaconis and Freedman \cite{DiaFree80} provided finite extendibility.
They show that, 
given a certain finite extendibility,
one may bound the total variation distance between the 
given exchangeable sequence and the closest (true) mixture of i.i.d.\ 
random variables
(and this provides another proof of de Finetti's theorem.)


Regarding infinite extendibility,
de Finetti \cite{DeFi31} gives a condition 
for the binary case ($|S|=2$) using characteristic functions. 
When $S$ is a general measurable space, no 
criteria for finite or infinite extendibility exist. 
For either problem, only {\em necessary} conditions exist; see, e.g.,
Scarsini \cite{scar85}, von Plato \cite{plato91}, and 
Scarsini and Verdicchio \cite{scarver93}. 
For the case $S=\R$ and when variances exist,
one simple necessary (but far from sufficient) condition 
for infinite extendibility is that any pair of variables 
have nonnegative covariance (\cite[page 591]{KerSze06}). 
This is certainly not sufficient 
(see \textsection A.1 for a simple counterexample).  
As far as we know, the extendibility problem  has been dealt with
on a case-by-case basis.
For example,
Gnedin \cite{Gnedin96} considers densities on $\R^n$ that are functions 
of minima and/or maxima, and Liggett, Steif and T\'oth \cite{chalmers07} 
solve a particular problem
of infinite extendibility within the context of statistical mechanical systems.


In this paper, we provide a necessary and sufficient condition for the
extendibility problem.
First of all, we do not restrict ourselves to the finite $S$ case.
One of our concerns is to work with as general $S$ as possible.
There are topological restrictions to be imposed on $S$, arising
from the methods of our proofs. 
We define certain linear operators via symmetrizing functionals
on finite products of $S$ and use functional analysis techniques
to give a necessary and sufficient condition for extendibility.

We use the term ``primitive $N$-extending functional'' for 
the most basic of these operators (see Section \ref{secpef}),
denote it by $\mathcal E^N_n$,
and construct it as follows.
Let $g: S^n\to \R$ be bounded and measurable. For $N \ge n$,
define $U^N_n g: S^N\to \R$ to be a symmetric function
obtained by selecting $n$ of the $N$ variables $x_1,\ldots,x_N$
at random without replacement and by evaluating $g$ at this selection:
\[
U^N_n g(x_1,\ldots,x_N) 
:= \frac{1}{(N)_n} \sum_{\sigma \in \S[n,N]}
g(x_{\sigma(1)}, \ldots, x_{\sigma(n)}),
\]
where 
\[
\S[n,N]:=\text{ the set of all injections from $\{1,\ldots,n\}$
into $\{1,\ldots,N\}$}
\]
and 
\[
(N)_n := N(N-1)\cdots (N-n+1) = |\S[n,N]|
\]
is the cardinality of this set.
When $N=n$, the set
$
\S[N]:=\S[N,N]
$
is the set of all permutations of $\{1,\ldots,N\}$.
Now define the linear functional
\[
\mathcal E^N_n: U^N_n g \mapsto \E g(X_1,\ldots,X_n),
\]
whenever the function $g:S^n \to \R$ is bounded and measurable. 
At first sight, this might not even be a function; but it turns out to be,
and this is shown in Section \ref{symmetrizing}, Lemma \ref{pcrucial}, 
a statement that depends on the properties of urn measures developed 
in Section \ref{urnmeasures}.
Let $(X_1,\ldots, X_n)$ be an exchangeable random element of $S^n$.
Define
\[
\|\mathcal E^N_n\| := \sup\bigg\{| \E g(X_1,\ldots,X_n) |:\,
\text{$g$ bounded measurable},\,
|U^N_n g(x)| \le 1 \text{ for all } x \in S^N
\bigg\}.
\]
Then our first result can be formulated as follows.
\begin{theorem}\label{ke}
Suppose that $S$ is a locally compact Hausdorff space and $\SS$ its
Borel $\sigma$-algebra.
Let $X=(X_1, \ldots, X_n)$ be an exchangeable random element of $S^n$
such that the law of $X_1$ is outer regular and tight. 
Let $N$ be a positive integer, $N \ge n$.
Then $(X_1,\ldots, X_n)$ is $N$-extendible if and only if $\|\mathcal E_{n}^N\|=1$.
\end{theorem}
The problem of infinite extendibility is addressed next.
That is, given an $n$-exchangeable probability measure that is
$N$-extendible for all $N > n$, is it true that it is infinitely-extendible?
The answer to this may seem ``obviously yes'' and it is so if
$S$ is a Polish space. In the absence of metrizability,
we work with quotient spaces in order to give an affirmative
answer to this question: see Theorem \ref{ine} in Section \ref{infext}.

The paper is organized as follows.
Section \ref{urnmeasures} develops some notation and results related
to urn measures. The main result of this section,
Lemma \ref{uinversion}, is responsible for the
fact that $\mathcal E^N_n$ is well-defined and could be of 
independent interest. 
In Section \ref{symmetrizing} we relate symmetrization operations
to urn measures and are naturally led to the construction of $\mathcal E^N_n$
and its extensions that we call extending functionals.
These are functionals that extend $\mathcal E^N_n$ on the space
$b(S^N)$ of bounded measurable functions $f: S^N \to \R$ and are
symmetric in the sense that their action on $f$ does not depend
on the order of arguments of $f$. Their existence is guaranteed by
the Hahn-Banach theorem.
We then develop several properties of symmetrization
operations and extending functionals
(Lemmas \ref{Uu}--\ref{necessary}).
Lemma \ref{necessary}, in particular, provides the necessity part
of Theorem \ref{ke} and requires no assumptions on $S$.
The results of \textsection\textsection \ref{secsym}, \ref{secpef}, \ref{secext}
 require no assumptions on $S$ neither.
With the view towards establishing the sufficiency part of Theorem \ref{ke}
we assume that $S$ is a locally compact Hausdorff space in \textsection
\ref{critfe} and prove the sufficiency part in this section.
Section \ref{infext} deals with infinite exchangeability, under the same
assumptions on $S$. 
Section \ref{true} gives a condition under which the signed measure
$\nu$ in the representation formula \ref{fre} is a probability measure.
In Section \ref{addi} we give a different version of the main theorem
and some results on persistence of extendibility property
under limits.

\section{Urn measures}
\label{urnmeasures}
Let $S$ be a set and $N$ a positive integer. A point measure $\nu$
is any measure of the form $\nu = \sum_{i=1}^d c_i \delta_{a_i}$,
for $c_i$ nonnegative integers and $a_i \in S$,
where $\delta_a(B)$ is $1$ if $a \in B$, and $0$ if $a \not \in B$,
for $B \subset S$. 
We let $\NN(S)$ be the
collection of all point measures. We let $\NN_N(S)$ be the set
of all $\nu \in \NN(S)$ with total mass $N$.
Point measures are defined on all subsets of $S$.
We write $\nu\{a\}$ for the value of $\nu$ at the
set $\{a\}$ containing the single point $a \in S$.
If $\mu, \nu$ are two point measures, we write $\mu \le \nu$
whenever $\mu(B) \le \nu(B)$ for all $B \subset S$.
If $x =(x_1,\ldots, x_N) \in S^N$ we define its {\it type} to be the
point measure
\[
\epsilon_x := \sum_{i=1}^N \delta_{x_i}.
\]
For $\nu \in \NN_N(S)$, let
\[
S^N(\nu) := \{x \in S^N:\, \epsilon_x=\nu\}.
\]
This is a finite set whose cardinality is denoted by $\binom{N}{\nu}$
and which is easily seen to be given by
\[
\binom{N}{\nu} = \frac{N!}{\prod_{a\in S} \nu\{a\}!}.
\]
The product in the denominator involves only finitely many terms
different from $1$.
Let
\[
u_{N, \nu} := \text{the uniform probability measure on $S^N(\nu)$}.
\]
Let now $n \le N$ be a positive integer, consider the projection
\[
\pi^N_n: (x_1,\ldots,x_n, \ldots, x_N) \mapsto (x_1,\ldots,x_n),
\]
and let
\[
u^N_{n,\nu} := \text{image of $u_{N, \nu}$ under $\pi^N_n$}.
\]
Clearly, $u^N_{N,\nu} = u_{N,\nu}$.
\footnote{The measure $u^N_{n,\nu}$ is called urn measure because of its
probabilistic interpretation: think of the elements of the support
of $\nu$ as colors and consider an urn containing $N$ balls
such that there
are $\nu\{a\}$ balls with color $a$;
make $n$ draws without replacement; then $u^N_{n,\nu}$ is
the probability distribution of the colors drawn.}
The support of $u^N_{n,\nu}$ is the (finite) set of
all $(x_1,\ldots,x_n) \in S^n$ such that 
$\delta_{x_1}+\cdots+\delta_{x_n} \leq \nu$.
If $(x_1,\ldots,x_N) \in S^N(\nu)$ and $(x_1,\ldots,x_n) \in S^n(\mu)$,
where $\mu \le \nu$, we can easily compute the value of $u^N_{n,\nu}$
at the set containing the single point $(a_1,\ldots,a_n) \in S^n(\mu)$
as follows:
\begin{align*}
u^N_{n,\nu}\{(a_1,\ldots,a_n)\} 
&= u_{N,\nu} \{(x_1,\ldots,x_N) \in S^N(\nu):\, x_1=a_1,\ldots,x_n=a_n\}
\\
&= u_{N,\nu} \{(a_1,\ldots,a_n,x_{n+1},\ldots x_N) \in S^N:\, 
\delta_{x_{n+1}}+\cdots+\delta_{x_N} = \nu-\mu\}
\\
&= \frac{|S^{N-n}(\nu-\mu)|}{|S^N(\nu)|}
= \frac{\binom{N-n}{\nu-\mu}}{\binom{N}{\nu}}.
\end{align*}
But $u^n_{n,\mu}$ is the uniform probability measure on $S^n(\mu)$;
so $u^n_{n,\mu}\{(a_1,\ldots,a_n)\} = 1/\binom{n}{\mu}$, and so
the above can be written as
\begin{equation}
\label{udir}
u^N_{n,\nu} = \sum_{\mu \in \NN_n(S)} a(\nu,\mu) \, u^n_{n,\mu},
\end{equation}
where
\begin{equation} \label{mhg}
a(\nu,\mu) := \binom{n}{\mu} \binom{N-n}{\nu-\mu} \bigg/
\binom{N}{\nu},
\end{equation}
with the understanding that $a(\nu,\mu)=0$ if it is not the case
that $\mu \le \nu$.
Simple algebra then shows that
\begin{equation}
\label{udir2}
a(\nu,\mu) =
\frac{\prod_{b \in S} \binom{\nu\{b\}}{\mu\{b\}}}{\binom{N}{n}},
\quad
\nu \in \NN_N(S),\, \mu \in \NN_n(S),\, \mu \le \nu.
\end{equation}
The binomial symbols
in \eqref{udir2} are now standard ones: $\binom{N}{n} = N!/n!(N-n)!$.
We now pass on to the following algebraic fact stating
that for any $\mu \in\NN_n(S)$, and any $N \ge n$, we can express $u^n_{n,\mu}$
as a linear combination of finitely many urn measures $u^N_{n,\nu}$.
\begin{lemma}
\label{uinversion}
Fix $n, N$ be positive integers, $n \le N$
and  $\mu \in \NN_n(S)$. 
Then there exist $c(\mu,\nu) \in \R$, $\nu \in \NN_N(S)$,
such that $c(\mu,\nu)$ is zero for all but finitely many $\nu$ and
\begin{equation}\label{lambdau}
u^n_{n,\mu} = \sum_{\nu \in \NN_N(S)} c(\mu,\nu)\, u^N_{n,\nu}.
\end{equation}
Moreover there exists $K>0$ (depending on $n$ and $N$) such that
\begin{equation}\label{k}
\sup_{\mu \in \NN_n(S)}\sum_{\nu \in \NN_N(S)}|c(\mu,\nu)|<K.
\end{equation}
\end{lemma}
\begin{proof}
Let $\mu \in \NN_n(S)$ and let $T_\mu$ be its support.
Let $k$ be the cardinality of $T_\mu$. Write  $T_\mu=\{a_1,\ldots,a_k\}$ 
and order it in an arbitrary way, say:
\[
a_1 < \cdots < a_k.
\]
This order induces an order on $\NN_n(T_\mu)$, the set of point measures
of mass $n$ supported on $T_\mu$: for $\lambda_1, \lambda_2 \in \NN_n(T_\mu)$,
we write
\[
\lambda_1 \prec \lambda_2
\]
if $(\lambda_1\{a_1\}, \ldots, \lambda_1\{a_{k}\})$
is lexicographically smaller than
$(\lambda_2\{a_1\}, \ldots, \lambda_2\{a_{k}\})$, that is,
\[
\exists~ 1 \le j \le k \text{ such that }
\lambda_1\{a_i\}= \lambda_2\{a_i\},\, 1\le i \le j-1
\text{ and } 
\lambda_1\{a_j\} < \lambda_2\{a_j\}.
\]
This is a total order.
For $\lambda\in \NN_n(T_\mu)$, let $\lambda^N \in \NN_N(S)$ be defined by
\[
\lambda^N = \lambda + (N-n) \delta_{a_k}.
\]
By \eqref{udir},
\[
u^N_{n,\lambda^N} = \sum_{\kappa \in \NN_n(T_\mu)} a(\lambda^N,\kappa) \, u^n_{n,\kappa},
\quad \lambda \in \NN_n(T_\mu).
\]
Observe that the square matrix
$[a(\lambda^N,\kappa)]_{\lambda,\kappa \in \NN_n(T_\mu)}$
has a lower triangular structure with respect
to the lexicographic order,
\[
a(\lambda^N,\kappa)=0 \text{ if } \lambda \prec \kappa,
\]
while 
\[
a(\lambda^N,\lambda)>0.
\]
These follow from \eqref{udir2} and the definitions above.
Therefore the square matrix is invertible and the claim follows.
From the above construction, $\{c(\mu,\nu)\}_{\nu\in \NN_n(S)}$ depends only
on the cardinality of the support of $\mu$ and mass on each atom
in the support. 
So (\ref{k}) follows. 
\end{proof}


\section{Symmetrizing operations and extending functionals}
\label{symmetrizing}
\subsection{Symmetrization of a function.}
\label{secsym}
We say that a function $g(x_1,\ldots, x_k)$ is symmetric
if it is invariant under all $k!$ permutations of its arguments.
From a real-valued function $g(x_1,\ldots,x_n)$ on $S^n$ we produce 
a symmetric function, denoted by $U^N_n g$, on $S^N$, by:
\begin{equation}
\label{UNn}
U^N_n g(x_1,\ldots, x_N) = \frac{1}{(N)_n}
\sum_{\sigma \in \S[n,N]} g(x_{\sigma(1)}, \ldots, x_{\sigma(n)})
\end{equation}
and view $U^N_n$ as a linear operator from the space of
(bounded measurable) functions on $S^n$ into the space
of (bounded measurable) symmetric functions on $S^N$.
To link it to urn measures, let $U^N_{n,x}$ be the probability 
measure on $S^N$ obtained by making $n$ ordered selections without
replacement from an urn containing $N$ balls labelled $x_1, \ldots, x_N$,
at random:
\[
U^N_{n,x} := \frac{1}{(N)_n} 
\sum_{\sigma \in G[n,N]} \delta_{x_{\sigma(1)}, \ldots, x_{\sigma(n)}}.
\]
Thus, $U^N_{n,x}$ is an exchangeable probability measure on $S^n$ 
for all $x \in S^N$.
Clearly,
\[
U^N_n g (x_1,\ldots,x_N) = \int_{S^n} g\, dU^N_{n,x}, 
\quad x=(x_1,\ldots,x_N) \in S^N.
\]
But we observe that 
\begin{lemma}
\label{Uu}
$U^N_{n,x}$ is the same as $u^N_{n, \epsilon_x}$.
\end{lemma}
\begin{proof}
To see this, use \eqref{udir} and \eqref{udir2} to write
\begin{align*}
u^N_{n, \nu} &= \sum_{\mu\in \NN_n(S)}
\frac{\prod_{a \in S} \binom{\nu\{a\}}{\mu\{a\}}}{\binom{N}{n}} \,
\frac{1}{\binom{n}{\mu}} \sum_{y \in S^n(\mu)} \delta_y,
\\
&= \sum_{\mu\in \NN_n(S)} \frac{\prod_a (\nu\{a\})_{\mu\{a\}}}{(N)_n}
\sum_{y \in S^n(\mu)} \delta_y
\\
&= \sum_{y\in S^n} \sum_{\mu\in \NN_n(S)} \frac{\prod_a (\nu\{a\})_{\mu\{a\}}}{(N)_n}\,
\1\{\epsilon_y=\mu\}\, \delta_y
\\
&= \sum_{y\in S^n}  \frac{\prod_a (\nu\{a\})_{\epsilon_y\{a\}}}{(N)_n}\, \delta_y.
\end{align*}
On the other hand, if $\sigma^*$ is a random element of $\S[n,N]$
with uniform distribution then, for fixed $x=(x_1,\ldots,x_N) \in S^N$
and $y=(y_1,\ldots,y_n)\in S^n$,
\[
U^N_{n,x}\{y\} 
= P\{x_{\sigma^*(1)} = y_1, \ldots, x_{\sigma^*(n)} = y_n\},
\]
which is zero unless $y_i \in \{x_1,\ldots,x_n\}$, that is when
$\epsilon_y \le \epsilon_x$, in which case
\[
P\{x_{\sigma^*(1)} = y_1, \ldots, x_{\sigma^*(n)} = y_n\}
=\frac{\prod_a (\epsilon_x\{a\})_{\epsilon_y\{a\}}}{(N)_n},
\]
and this agrees with $u^N_{n,\nu} \{y\}$ when $\nu=\epsilon_x$.
\end{proof}
We summarize some properties of the operator $U^N_n$ below. 
We denote by $\|g\|$ the sup-norm
of a real-valued function $g: S^n \to \R$.
\begin{lemma} 
Let $g: S^n \to \R$ be a function, $U^N_n$ as in \eqref{UNn},
and $f:S^N \to \R$ a symmetric $\SS^N$-measurable function. Then
\label{pp}
\begin{itemize}
\item[(i)] $\|U^N_n g \| \le \|g\|$.
\item[(ii)] $U^{n_3}_{n_1} = U^{n_3}_{n_2} \comp U^{n_2}_{n_1}$, if
$n_1 \le n_2 \le n_3$.
\item[(iii)] $\|U^N_n g\|$ decreases as $N$ increases.
\item[(iv)] 
The function $f$ is also measurable with respect to the $\sigma$-algebra generated
by $\{U^N_n h\}$ where $h$ ranges over all measurable function from
$S^n$ into $\R$.
\end{itemize}
\end{lemma}
\begin{proof}
(i) Immediate from the definition.
\\
(ii) It follows from the fact that $u^N_{n,\nu}$ is the image of
the uniform measure on $S^N(\nu)$ under the projection $S^N \to S^n$.
\\
(iii) It follows from (i) and (ii).
\\
(iv)
It suffices to show this for $g(x_1,\ldots,x_N)= \1_B(x_1) \cdots \1_B(x_N)$,
where $B \in \SS$.
Let $f=U^N_n \1_{B^n}$. Then $g=1$ if and only if $f=1$.
Hence $g$ is a measurable function of $f$.
\end{proof}

If we take a measurable function $g: S^n \to \R$ whose symmetrization
$U^n_n g$ is identically zero then, clearly, $g(X_1,\ldots,X_n)=0$, a.s.,
if $X=(X_1,\ldots,X_n)$ has exchangeable law.
Now suppose $N > n$ and assume $U^N_n g(x_1,\ldots,x_N)$ is the identically
zero function. Again, the conclusion that $g(X_1,\ldots,X_n)=0$, a.s.,
holds true. This is the content of the next lemma.
\begin{lemma}
\label{pcrucial}
Let $g:S^n \to \R$ be a measurable function such that, for some $N \ge n$,
\[
U^N_n g(x_1,\ldots, x_N) = 0, \text{ for all } (x_1,\ldots,x_N)\in S^N.
\]
Then, if $X=(X_1,\ldots,X_n)$ is an exchangeable random element of $S^n$,
we have 
\[
g(X_1,\ldots,X_n) = 0, \text{ a.s.}
\]
\end{lemma}
\begin{proof}
Since $U^N_n g(x)$ is obtained by integrating $g$ with
respect to the measure $U^N_{n,x}$, the assumption that $U^N_n g$
is identically equal to zero implies, due to Lemma \ref{Uu}, that
\[
\int_{S^n} g\, du^N_{n,\nu} = 0, \text{ for all $\nu \in \NN_N(S)$}.
\]
By Lemma \ref{uinversion} this implies that
\[
\int_{S^n} g\, d u^n_{n,\mu}=0, \text{ for all $\mu \in \NN_n(S)$}.
\]
Let $x=(x_1,\ldots,x_n)$.
Set $\mu=\epsilon_x$ in the last display and use again Lemma \ref{Uu}
to obtain
\[
U^n_{n,x} g =0, \text{ for all } x \in S^n.
\]
Suppose now that $X=(X_1,\ldots,X_n)$ has exchangeable law. 
Then
\[
g(X) \eqdist U^n_{n,X} g =0,
\]
and so $g(X)=0$, a.s.
\end{proof}

\subsection{The primitive extending functional.}
\label{secpef}
We come now to the definition of the main object of the paper.
Given an $n$-exchangeable $X=(X_1,\ldots, X_n)$, and $N \ge n$,
the result of Lemma \ref{pcrucial} tells us that the assignment
\begin{equation}
\label{pef}
\mathcal E^N_n : U^N_n g \mapsto \E g(X_1,\ldots, X_n),
\end{equation}
for $g: S^n \to \R$ bounded measurable,
is a well-defined function.
Indeed, if $U^N_n g = U^N_n h$ then $U^N_n (g-h)=0$ and
so, by Lemma \ref{pcrucial}, $\E g(X)=\E h(X)$.
Let $b(S^n)$ be the Banach space of bounded measurable real-valued functions
$g: S^n \to \R$ equipped with the sup norm.
We call $\mathcal E^N_n$ {\em primitive extending functional}.
The norm of $\mathcal E^N_n$ induced by the sup norm on $b(S^n)$ is
\begin{equation}
\label{Enorm}
\|\mathcal E^N_n\| =  \sup_{g \in b(S^n), g \neq 0}
\frac{|\E g(X_1,\ldots,X_n)|}{\|U^N_n g\|}.
\end{equation}
\begin{lemma}
$\mathcal E^N_n$ is a bounded linear map from $U^N_n(b(S^n))$ into $\R$
with norm at least $1$.
\end{lemma}
\begin{proof}
Linearity is immediate.
Boundedness of $\mathcal E^N_n$ is tantamount to the existence of $K<\infty$
such that $|\E g(X)| \le K \|U^N_n g\|$ for all bounded measurable
$g:S^n\to \R$. 
By exchangeability and Lemma \ref{Uu} we have 
\[
\E g(X) = \E U^n_n g(X) = \E \int_{S^n} g\, d u^n_{n,\epsilon_X}.
\]
By Lemma \ref{uinversion}, 
\[
\left|\int_{S^n} g\, d u^n_{n,\mu}\right|
= \left|\sum_\nu c(\mu,\nu)\, \int_{S^n} g\, du^N_{n,\nu}\right|
\le \|U^N_n g\| \, \sum_\nu |c(\mu,\nu)|.
\]
The reason for the latter inequality is that 
\[
\sup_{\nu \in \NN_N(S)} \left| \int_{S^n} g\, du^N_{n,\nu}\right|
= \sup_{x\in S^N} \left| \int_{S^n} g\, dU^N_{n,x}\right|
= \sup_{x\in S^N} |U^N_n g(x)| = \|U^N_n g\|.
\]
Therefore, $|\E g(X)| \le \E \sum_\nu |c(\epsilon_X,\nu)| \, 
\|U^N_n g\|< K\, \|U^N_n g\|$ with $K$ in (\ref{k}).
To see that $\|\mathcal E^N_n \| \ge 1$, quite simply notice that
for $g(x_1,\ldots,x_n)=1$ the ratio inside the supremum 
in \eqref{Enorm} equals 1.
\end{proof}

That extendibility is captured by $\mathcal E^N_n$ is a consequence of
the following two simple lemmas. The first is a straightforward
rewriting of the definition of extendibility
\cite[Prop.\ 1.4]{spi82}.
\begin{lemma}
\label{extcharaprop}
Fix $N \ge n$.
An exchangeable random element $X=(X_1,\ldots,X_n)$  of $S^n$
is $N$-extendible if and only if there is an exchangeable probability
measure $Q$ on $S^N$ such that
\begin{equation}
\label{extchara}
\E g(X_1,\ldots,X_n) = \int_{S^N} (U^N_n g) \, dQ,
\end{equation}
for all bounded measurable $g: S^n \to \R$.
\end{lemma}
\begin{proof}
If $X=(X_1,\ldots,X_n)$ is $N$-extendible there is exchangeable
random element $Y=(Y_1,\ldots,Y_N)$ of $S^N$ such that
$(X_1,\ldots,X_n) \eqdist (Y_1,\ldots,Y_n) 
\eqdist (Y_{\sigma(1)},\ldots,Y_{\sigma(n)})$ for 
all $\sigma \in \S[n,N]$.
Therefore \eqref{extchara} holds with $Q$ the law of $Y$.
Conversely, if \eqref{extchara} holds, let
$(Y_1,\ldots, Y_N)$ be a random element of $S^N$ with law $Q$.
Since $Q$ is exchangeable, the right-hand side of \eqref{extchara}
is equal to $\E g(Y_1,\ldots, Y_n)$. Hence
$\E g(X_1,\ldots,X_n) = \E g(Y_1,\ldots, Y_n)$
for all bounded measurable $g: S^n \to \R$ and this means that
$X$ is $N$-extendible.
\end{proof}
The second lemma is a necessary condition for extendibility
in terms of the norm of $\mathcal E^N_n$:
\begin{lemma}\label{necessary}
If the exchangeable random element $X=(X_1,\ldots,X_n)$ of $S^n$
is $N$-extendible then $\|\mathcal E^N_n\|=1$, where
$\mathcal E^N_n$ is the functional defined by \eqref{pef}.
\end{lemma}
\begin{proof}
Assuming $N$-extendibility, by Lemma \ref{extcharaprop} there is
an exchangeable probability measure $Q$ on $S^N$ such that
\eqref{extchara} holds. Then
\[
|\E g(X)| = \bigg|\int_{S^N} (U^N_n g)\, d Q\bigg|
\le\|U^N_n g\|,
\]
for all bounded measurable $g:S^n \to \R$.
Thus $|\mathcal E^N_n f| \le \|f \|$ for all $f$ in the domain
of $\mathcal E^N_n$. So $\|\mathcal E^N_n \|\le 1$.
Since $\|\mathcal E^N_n \|\ge 1$, we actually have equality.
\end{proof}


\subsection{Extending functionals and their properties.}
\label{secext}
Throughout, $n, N$ are fixed positive integers, $n \le N$,
and $X=(X_1,\ldots,X_n)$ is an exchangeable random element of $S^n$.
The primitive extending functional of \eqref{pef} depends on $n$, $N$
and the law of $(X_1,\ldots,X_n)$. As $N$ and $n$ are for now kept 
fixed we denote $\mathcal E^N_n$ simply by $\mathcal E$.
We pass on to a more general object than  $\mathcal E$.

An operator defined on $b(S^N)$ is called symmetric
if, whenever $\sigma$ is a permutation of $\{1,\ldots,N\}$
and $f \in b(S^N)$, its value at $\sigma f$ does not depend on $\sigma$,
where $\sigma f(x_1,\ldots,x_N)= f(x_{\sigma(1)},\ldots, x_{\sigma(N)})$.

\begin{definition}
We call extending functional
any symmetric linear functional
$\mathcal L : b(S^N) \to \R$  agreeing with $\mathcal E$ on the domain
of the latter and such that $\|\mathcal L \| = \|\mathcal E\|$. 
\end{definition}
We shall need the following simple lemma.
When $V$ is a linear subspace of $b(S^N)$ and $f\in V$ we write $\|f\|_V =
\sup_{x \in S^N} |f(x)|$ and if $\LL:V \to \R$, we write 
$\|\LL\|_V = \sup_{\|f\|_V \le 1} |\LL(f)|$.
\begin{lemma}
\label{Lsimple}
Let $V$ be a linear subspace of $b(S^N)$ containing constant functions
and $\LL: V \to \R$ a linear functional such that $\LL(1)=1$ and
$\|\LL\|_V=1$. Then $\LL$ is monotone: If $f, g \in V$, $f \le g$ pointwise,
then $\LL f \le \LL g$.
\end{lemma}
\begin{proof}
It suffices to show that if $f \ge 0$, $f \in V$ then $\LL f \ge 0$.
For such an $f$ let $h:= \|f\|_V-f$.
Then $\LL h = \|f\|_V \LL(1) - \LL f = \|f\|_V- \LL f$.
But $|\LL h| \le \|\LL\|_V\, \|h\|_V = \|h\|_V$, and $\|h\|_V \le \|f\|_V$
because both $f$ and $h$ are nonnegative. Hence
$ \|f\|_V- \LL f \le \|f\|_V$ and so $\LL f \ge 0$.
\end{proof}


Some properties are summarized below.
Note that when $f:S^N\to \R$ and $t \in \R$ then 
$\1_{f\le t}$ stands for the function on $S^N$ that is $1$
on the set $\{x \in S^N:\, f(x) \le t\}$ and $0$ on its complement.
\begin{lemma}
\label{propertiesofL}
\begin{itemize}
\item[(i)]
Extending functionals exist and any extending functional $\LL$
satisfies $\LL(1)=1$;
\item[(ii)]  
if $\|\LL\|=1$ then $\LL$  is monotone.
\item[(iii)]
if $\|\LL\|=1$ and $f$ a function such that 
$f(x)\le 1$ for all $x \in S^N$ then
\[
\LL\1_{f\le t} \le \frac{1-\LL f}{1-t}, \quad t < 1.
\]
\end{itemize}
\end{lemma}
\begin{proof}
(i) Since the domain $U^N_n b(S^n)$ is a linear subspace of the Banach space
$b(S^N)$, it follows from the Hahn-Banach theorem \cite[Theorem 6.1.4]{DUDLEY02}
that there is $\widetilde{\mathcal E}: b(S^N) \to \R$ extending $\mathcal E$
and such that $\|\widetilde{\mathcal E}\|=\|\mathcal E\|$.
To produce a symmetric functional from this extension, set
\[
\mathcal L := \widetilde{\mathcal E} \comp U^N_N.
\]
Since $\mathcal L$ is an extension of $\mathcal E$ we  have
$\|\mathcal L \| \ge \|\mathcal E \|$. On the other hand,
$\|\mathcal L \| \le \|\widetilde{\mathcal E}\| \, \|U^N_N\|
= \|\widetilde{\mathcal E}\| = \|\mathcal E \|$.
Hence $\|\mathcal L \|=\|\mathcal E \|$.
Arguing as in the last step of the proof of Lemma \ref{necessary},
we have $\1_{S^N} = U^N_n \1_{S^n}$ and so, from the definition of 
$\mathcal E$, we have $\mathcal E(U^N_n \1_{S^n})= \E \1_{S^n}(X)=1$.
Since $\mathcal L$ extends $\mathcal E$, we indeed have $\LL(1)=1$. 
\\
(ii) Lemma \ref{Lsimple} applies.
\\
(iii) Suppose $f(x_1,\ldots,x_N) \le 1$ for all $(x_1,\ldots,x_N)
\in S^N$. Then
\[
f+(1-t) \1_{f \le t} \le \1_{S_N}
\]
and so
\[
\LL(f)+(1-t) \LL(\1_{f\le t}) = \LL(f+(1-t) \1_{f \le t}) \le 1.
\]
\end{proof}
The key to constructing an $N$-extension of $(X_1,\ldots,X_n)$
are the properties of the set function
\[
F_\LL(A) := \LL(\1_A), \quad A \in \SS^N,
\]
particularly on measurable rectangles. A measurable rectangle is a subset 
of $S^N$ of
the form $B_1\times \cdots \times B_N$, with
$B_1,\ldots, B_N \in \SS$.
The reason we insist on rectangles is because of symmetrization operations.
This will become clear in the proof of Theorem \ref{ke}, 
in Section \ref{critfe}.
For now we show the following.
\begin{lemma}
\label{CaraLike}
Let $X=(X_1,\ldots,X_n)$ be an exchangeable random element of $S^n$.
If $\LL$ is an extending functional such that $\|\LL\|=1$ and
$F_\LL$ is countably additive on the algebra generated
by measurable rectangles of $S^N$ then there exists a unique exchangeable
probability measure $Q$ on $S^N$ such that $Q(A)=F_\LL(A)$ for measurable
rectangles $A \subset S^n$ and such that $Q$ is an $N$-extension
of the law of $X$.
\end{lemma}
\begin{remark}
\label{Lnotmeas}
Before proving this, we point out that even under the condition
that $\|\LL\|=1$ (which will turn out to be sufficient for extendibility),
the function $F_\LL(\cdot)$ need not be a measure. What the last lemma says
is that there is a probability measure $Q$ agreeing with $F_\LL$ on 
the algebra of rectangles but that $Q(A)$ is not necessarily equal to
$F_\LL(A)$ for arbitrary measurable $A \subset S^N$.
We refer to \textsection A.2 for an example.
\end{remark}

\begin{proof}[Proof of Lemma \ref{CaraLike}]
By linearity of $\LL$, $F_\LL$ is finitely additive.
By (ii) of Lemma \ref{propertiesofL}, $F_\LL(A) \ge 0$.
By the countable additivity assumption on the algebra of rectangles and 
Carath\'eodory's extension theorem, there is a probability measure
$Q$ on $S^N$ agreeing with $F_\LL$ on rectangles. Since
$\LL$ is a symmetric functional, we have that $F_\LL(\sigma^{-1} A)
= F_\LL(A)$, where $\sigma^{-1} A = \{(x_{\sigma(1)}, \ldots, x_{\sigma(N)}):\,
(x_1,\ldots,x_N) \in A\}$, and $\sigma$ a permutation of
$\{1,\ldots, N\}$.
Hence $Q$ is an exchangeable  probability measure on $S^N$.
We show that $Q$ is an extension of the law of $X=(X_1,\ldots,X_n)$
by showing that \eqref{extchara} holds.
It suffices to show that it holds for $g=\1_A$ with $A \subset S^n$
is a rectangle, i.e., $A=A_1\times \cdots \times A_n$,
with $A_i \in \SS$.
But then $\sigma^{-1} A \subset S^N$ is in the algebra  of
measurable rectangles of $S^N$.
Starting from the right-hand side of \eqref{extchara} we have
\begin{align*}
\int_{S^N} (U^N_n \1_A)\, dQ
&= \frac{1}{(N)_n} \sum_{\sigma \in \S[n,N]} 
\int_{S^N} \1_{\sigma^{-1}A}\, dQ
\\
&= \frac{1}{(N)_n} \sum_{\sigma \in \S[n,N]} Q(\sigma^{-1} A)
\\
&\stackrel{(a)}{=} \frac{1}{(N)_n} \sum_{\sigma \in \S[n,N]} \LL(\1_{\sigma^{-1} A})
\\
&= \LL\left( \frac{1}{(N)_n} \sum_{\sigma \in \S[n,N]} \1_{\sigma^{-1} A}\right)
\\
&= \LL(U^N_n \1_A)
\\
&\stackrel{(b)}{=} \mathcal E(U^N_n \1_A)
\\
&\stackrel{(c)}{=} \P(X \in A),
\end{align*}
where (a) is because $Q$ agrees with $F_\LL$ on the algebra of measurable
rectangles of $S^N$, (b) because $\LL$ agrees with $\mathcal E$ on the domain of
$\mathcal E$ and (c) by the definition \eqref{pef} of $\mathcal E$.
\end{proof}


\subsection{A criterion for finite extendibility.}
\label{critfe}
Up to now, we kept the space $S$ as general as possible.
We shall now need to introduce some topological
assumptions on $S$ and some assumptions on the given
probability measure.

We assume that $S$ is a locally compact Hausdorff space.
Then $\SS$ is the class of its Borel sets, the smallest
$\sigma$-algebra containing open sets.
A probability measure $P$ on $S$ is tight if for all $\epsilon > 0$
there is a compact set $K$ such that $P(K) \ge 1-\epsilon$.
A measure $P$ is outer regular\footnote{
For probability measures, 
outer regularity and tightness are equivalent to regularity 
(i.e., inner regularity and outer regularity). 
But we
use the former terminology throughout the paper, 
since we will use explicitly outer regularity
and tightness in the proofs. }
if for all $A \in \SS$,
$P(A) = \inf P(O)$ where the infimum is taken over all open sets $O \supset A$.
\begin{lemma}
\label{Lboundedcompacts}
Suppose that $S$ is a locally compact Hausdorff topological space
and $X=(X_1,\ldots,X_n)$ an exchangeable random element of $S^n$
such that the law of $X_1$ is tight. 
Fix $N > n$. 
If $\LL$ is an $N$-extending functional with norm $1$ then
the restriction of $\LL$ on the space $C_c(S^N)$ of continuous
functions with compact support has norm $1$ also.
\end{lemma}
\begin{proof}
Let $0 < \epsilon < 1/n$. By tightness, there is a compact set $K \subset S$
such that $\P(X_1 \in K) \ge 1-\epsilon$. Hence
$\P((X_1,\ldots,X_n) \in K^n) \ge 1-n\epsilon$. 
Let
\[
g(x_1,\ldots,x_n) := \1_{K^n}(x_1,\ldots,x_n).
\]
It is easy to see that 
\[
t:= \sup_{x \not \in K^n} U^N_n g(x) <1.
\]
Apply (iii) of Lemma \ref{propertiesofL} to $f=U^N_n g$:
\[
\LL(\1_{U^N_n g \le t}) \le \frac{1-\LL(U^N_n g)}{1-t}.
\]
But
\[
\L(\1_{U^N_n g \le t}) = \LL(\1_{S^N\setminus K^N})
=1-\LL(\1_{K^N}),
\]
while
\[
\LL(U^N_n g) = \mathcal E(U^N_n g) = \E g(X) = \P(X \in K^n) \ge 1-n\epsilon
\]
and so
\[
\L(\1_{K^N}) \ge 1-\frac{n\epsilon}{1-t}.
\]
By Urysohn's lemma \cite[p.\ 39]{rudin1987real},
there exists a function $F \in C_c(S^N)$ such that $0\le F \le 1$ everywhere
and $F=1$ on $K^N$. Hence $\1_{K^N} \le F$. By the monotonicity of $\LL$,
\[
\LL(F) \ge \LL(\1_{K^N}) \ge  1-\frac{n\epsilon}{1-t}.
\]
Since $\epsilon$ is an arbitrary positive number this says that $\LL(F) \ge 1$
implying that the norm of $\LL$ restricted to $C_c(S^N)$ is at least $1$.
On the other hand, the norm of the restriction cannot be larger than 
the norm of $\LL$ which is $1$. Hence the norm of the restriction is
equal to $1$.
\end{proof}

We now pass on to the proof of the main theorem. 

\proof[Proof of Theorem \ref{ke}]
By Lemma \ref{Lboundedcompacts}, the norm of the restriction of $\LL$ on
$C_c(S^N)$ is 1. By the Riesz representation theorem 
\cite[p.\ 40]{rudin1987real} there is a measure $Q$ on $(S^N, \SS^N)$
such that
\[
\L f = \int_{S^N} f\, dQ, \quad f \in C_c(S^N).
\]
By (i) and (ii) Lemma \ref{propertiesofL}, we have that $Q$ is 
a probability measure. Moreover, the Riesz representation theorem
guarantees that
\begin{align}
Q(G) 
&= \sup \{ \LL f:\, 
f \in C_c(S^N),\, 0 \le f \le 1, \, \operatorname{supp}(f) \subset G\},
\qquad &\text{ open } G \subset S^N,
\label{Ropen}
\end{align}
where $\operatorname{supp}(f)=\{x: f(x)\neq0\}$.
We shall prove that $Q$ provides the announced $N$-extension.
Fix a measurable rectangle
\[
R = A_1\times\cdots\times A_N,
\]
and $0< \epsilon < 1/N^2$. 
By the outer regularity of the law of $X_1$, pick open sets $A_{j,\epsilon}
\subset S$ such that
\begin{equation}
\label{uuu}
A_j \subset A_{j,\epsilon}, \quad
\P(X_1 \in A_j) \le 
\P(X_1 \in A_{j,\epsilon}) \le 
\P(X_1 \in A_j) +\epsilon, \quad 1 \le j \le N.
\end{equation} 
Then
\[
R \subset 
A_{1,\epsilon} \times \cdots \times A_{N,\epsilon} =: O_\epsilon
\]
and so
\begin{equation}
\label{OG}
G_\epsilon:= O_\epsilon \setminus R =
\{x \in S^N:\, x_j \in A_{j,\epsilon}\setminus A_j \text{ for some } 1\le j \le N\}
\end{equation}
Consider also the subset of $S^n$ defined by
\[
F_\epsilon := \{(x_1,\ldots,x_n) \in S^n:\,
x_i \in A_{j,\epsilon}\setminus A_j \text{ for some } 1\le i \le n, 1\le j \le N\}
\]
and, with $F_\epsilon^c = S^n \setminus F_\epsilon$, let
\[
f(x_1,\ldots,x_N) := (U^N_n \1_{F_\epsilon^c})(x_1,\ldots,x_N)
= \frac{\#\{\sigma \in \S[n,N]:\,
x_{\sigma(i)} \not\in A_{j,\,\epsilon}\setminus A_j,\,\forall i\le n, j \le N\}}
{(N)_n}
\]
If $x \in G_\epsilon$ then $x_j \in A_{j,\epsilon}\setminus A_j$
for some $j \le N$.
Then the number of injections $\sigma:\{1,\ldots,n\} \to \{1,\ldots,N\}$
such that $x_{\sigma(i)} \not\in A_{j,\epsilon}\setminus A_j$
for all $i=1,\ldots,n$ is at most the number of injections from
$\{1,\ldots,n\}$ into $\{1,\ldots,N\}\setminus\{j\}$, that is,
at most $(N-1)_n$. Therefore,
\[
G_\epsilon\subset \{x \in S^N:\, f(x) \le (N-n)/N\}.
\]
By the monotonicity of $\LL$,
\[
\LL(\1_{G_\epsilon}) \le \LL(\1_{f \le (N-n)/N}).
\]
We now apply (iii) of Lemma \ref{propertiesofL} to $f$, with $t=(N-n)/N$:
\[
\LL(\1_{f \le (N-n)/N}) \le \frac{1-\LL(f)}{1-(N-n)/N}.
\]
But
\begin{multline*}
\LL(f) = \LL(U^N_n \1_{F_\epsilon^c})
= \mathcal E(U^N_n \1_{F_\epsilon^c})
= \E \1_{F_\epsilon^c}(X_1,\ldots,X_n)
= \P((X_1,\ldots,X_n) \not \in F_\epsilon)
\\
\ge 1 - \sum_{i=1}^n \sum_{j=1}^N \P(X_i \in A_{j,\epsilon}\setminus A_j)
\ge 1-nN\epsilon,
\end{multline*}
where the last inequality is due to \eqref{uuu}.
Combining the above, we have
\[
\LL(\1_{G_\epsilon}) \le N^2\epsilon.
\]
Using \eqref{OG} and the monotonicity of $\LL$, we have
\[
\LL(\1_{O_\epsilon}) - \LL(\1_R)
= \LL(\1_{O_\epsilon\setminus R}) = \LL(\1_{G_\epsilon}) \le N^2\epsilon.
\]
On the other hand, by \eqref{Ropen} and the monotonicity of $\LL$,
\[
Q(O_\epsilon) \le \LL(\1_{O_\epsilon}).
\]
From the last two displays we have
\[
Q(R) \le Q(O_{\epsilon})\le \LL(\1_R) + N^2 \epsilon.
\]
Letting $\epsilon \downarrow 0$ we obtain
\[
Q(R) \le \LL(\1_R).
\]
This is true for all measurable rectangles $R$ and, by additivity,
true for all sets in the algebra of rectangles. Hence true for
$S^N\setminus R$. This implies that we actually have equality:
\[
\LL(\1_A) = Q(A),
\]
for all $A$  in the algebra of measurable rectangles of $S^N$.
Since $Q$ is a measure, the assumptions of Lemma \ref{CaraLike} hold
and so $Q$ is an $N$-extension of the law of $(X_1,\ldots, X_n)$.
\qed

\begin{corollary}
Given an $n$-exchangeable probability measure $P_n$ on $S^n$
and $N > n$, such that  one-dimensional marginal of $P_n$ 
is tight and outer regular and $S$ a locally
compact Hausdorff space,
we can formulate the criterion for $N$-extendibility as follows:
\begin{align}
\label{ccrit}
&\forall \epsilon > 0
~ \forall g \in \Phi_n
~ \exists a_1, \ldots, a_N \in S\text{ such that }&\nonumber\\
&\left|\int_{S^n} g(x)\, P_n(dx)\right|
\le
\frac{1+\epsilon}{(N)_n}
\bigg| \sum g(a_{\sigma(1)}, \ldots, a_{\sigma(n)})\bigg|,&
\end{align}
where the sum is taken over all injections $\sigma:  [n]\to [N].$
\end{corollary}
Indeed, this says that $\|\mathcal E^N_n\| \le 1$ and, since
$\|\mathcal E^N_n\| \ge 1$ always, it expresses precisely the condition
of Theorem \ref{ke}.

\begin{remark}
\label{proofnew}
In the proof of Theorem \ref{ke} we actually showed that,
under the stated conditions, for any $N$-extending functional $\LL:S^N\to \R$
there is a probability measure $Q$ on $S^N$ such that
$\LL \1_A = Q(A)$ for $A$ in the algebra of measurable rectangles of $S^N$.
\end{remark}



\section{From finite to infinite extendibility}
\label{infext}
We are seeking conditions that enable us to
extend an $n$-exchangeable probability measure to
an exchangeable probability measure on $S^\N$ in the standard sense. 
It seems natural to posit that $X$ is $N$-extendible for
all $N \ge n$ if and only if $X$ is extendible to $S^\N$. 
One direction is clear: If $X$ is an exchangeable random element of
$S^\N$ then $(X_1,\ldots, X_n)$ is $N$-exchangeable for all $N \in \N$.
But the other direction is not {\em a priori} obvious
since we may have an $N'$-extension and an $N''$-extension of $X$,
for some $n < N' < N''$, but the $N''$-extension may not
be an extension of the $N'$-extension.  Even worse, the $N'$-extension 
might not be further extendible.

One may attempt to use Prohorov's theorem 
to prove the infinite extendibility by
obtaining an appropriate infinite exchangeable sequence. 
This is possible in a metric space. But we work with
a locally compact Hausdorff space (not necessarily metrizable).
For a locally compact Hausdorff space
there is a version of Prohorov's theorem
\cite{enc} 
whose conclusion is stated in terms of continuous functions with
compact support.
This class of functions is not big enough for our purposes.
Indeed, as in Lemma \ref{extcharaprop} 
the set of test functions required for $N$-extendibility
is $U_n^N g$, where $g$ ranges over bounded measurable functions on $S^n$,
and these functions
are merely bounded.
 To bypass this difficulty, we will rely on a functional
analytic approach (and use the Hahn-Banach theorem again)
in the theorem below.

\begin{theorem}\label{ine}
Let $n$ be a positive integer, $n \ge 2$.
Assume that the hypotheses in Theorem \ref{ke} hold.
The following are equivalent:
\\
(a) $X$ is $N$-extendible for all $N \ge n$.
\\
(b) There is a random element $Y=(Y_1,Y_2,\ldots)$ of $S^\N$
with exchangeable law such that $(X_1,\ldots,X_n) \eqdist (Y_1, \ldots, Y_n)$.
\end{theorem}
\begin{proof}
Only (a) $\Rightarrow$ (b) needs to be shown. For $k \in \N$,
let $\Phi_k=b(S^k)$ be the set of bounded measurable functions on $S^k$
equipped with the sup norm.
Let $\Phi^*$ be the set of all bounded measurable real-valued
functions $f(x_1,\ldots,x_N)$ on $S^N$ for some $N \ge n$:
\[
\Phi^* := \bigcup_{N \ge n} \Phi_N.
\]
If $k < \ell$ then $\Phi_k$ is naturally embedded into $\Phi_\ell$:
if $f \in \Phi_k$ then we can define $\widetilde f \in \Phi_\ell$
by $\widetilde f(x_1, \ldots, x_k, \ldots, x_\ell) := f(x_1, \ldots, x_k)$. 
We shall write $\Phi_k \subset \Phi_\ell$ for this embedding; this should 
be read in the sense that the image of $\Phi_k$ under $f \mapsto \widetilde f$
is contained in $\Phi_\ell$.
If $f \in \Phi^*$ then there is a $k\ge n$ such that $f \in \Phi_k$. 
The $N$-symmetrized version of $f$ is $U^N_k f$ as in \eqref{UNn}.
Since $\Phi_k \subset \Phi_\ell$ for $k\le \ell$, 
we can also consider $U^N_\ell f$ for $k\le \ell  \le N$.
We can easily see $U^N_\ell f = U^N_k f$.

We let $i_f$ be the minimum $N$ such that $f \in \Phi_N$.
We next define a relation $\sim$ on $\Phi^*$ by
\[
f \sim g \iff \exists N ~ U^N_{i_f} f = U^N_{i_g} g, \quad f, g \in \Phi^*, N\geq \max\{i_f,i_g\}.
\]
We see that $\sim$ is an equivalence relation. To check transitivity,
suppose $f \sim g$ and $g \sim h$. Then $U^N_{i_f} f = U^N_{i_g} g$
and $U^M_{i_g} g = U^M_{i_h} h$ for some $M$ and $N$.
Letting $L:= \max(M,N)$ and using the property (ii) of Lemma \ref{pp}
we have $U^L_{i_f} f = U^L_{i_g} g$ and $U^L_{i_g} g = U^L_{i_h} h$,
implying that $f \sim h$.
In particular, notice that  any $f \in \Phi^*$ is equivalent to
some symmetric function because
\begin{equation}
\label{ets}
f \sim U^N_k f, \text{ for all $N \ge k \ge i_f$}.
\end{equation}
From the discussion above and property (ii) of Lemma \ref{pp} 
we also see that
\begin{equation}
\label{simeq}
f \sim g \iff \exists k_0~ \text{ so that if } 
k_0 \le k \le N \text{ then } U^N_k f = U^N_k g.
\end{equation}
Let $\bm[f\bm]$ be the equivalence class of $f$:
\[
\bm[f\bm] := \{g \in \Phi^*:~ g \sim f\},
\]
and  let $\Phieq$ be the collection of equivalence classes:
\[
\Phieq:= \{\bm[ f\bm]:\, f \in \Phi^*\}.
\]
We can easily check using \eqref{simeq}
that if $f \sim f'$ and $g \sim g'$ then, for all $\alpha,\beta \in \R$,
$\alpha f + \beta g \sim \alpha f' + \beta g'$.
Hence we can define
\[
\alpha \bm[f\bm] + \beta \bm[g\bm] := \bm[ \alpha f + \beta g\bm],
\]
which means that $\Phieq$ is a linear space with origin $\bm[0\bm]$,
the set of functions equivalent to the identically zero function.

By Lemma \ref{pp}(iii), the norm $\|U^N_{i_g} g\|$ decreases as
$N$ increases, so we attempt to define a norm on $\Phieq$ by
\[
\|\bm[ g\bm]\| := \lim_{N \to \infty} \|U^N_{i_g}g\|
= \inf_{N\geq i_g} \|U^N_{i_g}g\|.
\]
First, it is clear that if $g \sim h$ then $\|\bm[ g\bm]\| 
= \|\bm[ h\bm]\|$; so $\bm[ g\bm] \mapsto \|\bm[ g\bm]\|$
is a well-defined function.
To see that the triangle inequality holds we use \eqref{ets} and \eqref{simeq}.
Let $g_1, g_2 \in \Phi^*$. Then we can choose $k$ so that for
$g_1 \sim U^{N}_k g_1$ and $g_2 \sim U^{N}_k g_2$,
for all large $N$.
Then $g_1+g_2 \sim U^{N}_k g_1+U^{N}_k g_2 
= U^N_k (g_1+g_2)$ and so
\begin{align*}
\|\bm[ g_1+g_2\bm]\| &= \inf_{N\geq k} \|U^N_k (g_1+g_2)\|
\\
&\le \inf_{N\geq k} (\|U^N_k g_1 \| + \|U^N_k g_1 \|)
\\
&= \lim_{N \to \infty} (\|U^N_k g_1 \| + \|U^N_k g_1 \|)
= \|\bm[ g_1 \bm]\| + \|\bm[ g_2 \bm]\|.
\end{align*}
To check positive definiteness we prove the following:
\begin{lemma}
\label{claiminproof}
If $g \in \Phi^*$ has $\|\bm[g\bm]\|=0$ 
and if $f\in \Phi_N$ is a symmetric function
such that $f \sim g$ then 
$f$ is identically zero.
\end{lemma}
\begin{proof}
Let $f\in \Phi_N$ be a symmetric function and $\pi$ a probability measure on 
$(S,\SS)$.
Then 
\[
\pi^N(f) := \int_{S^N} f\, d\pi^N 
=\int_{S^{N+M}} (U_N^{N+M}f)\, d \pi^{N+M}.
\]
So then
$\|\pi^N(f)\|\leq \|U_N^{N+M}f\|$. Note that 
$\lim_{M\to\infty} \|U_N^{N+M}f\| =
\|\bm[f\bm]\|$.
Let $g \in \Phi^*$ have $\|\bm[g\bm]\|=0$ and assume $f \sim g$. 
Then $\|\bm[f\bm]\|=\|\bm[g\bm]\|=0$.
Hence 
\[
\pi^N(f)=0, \quad\text{ for any probability measure $\pi$}. 
\]
By (\ref{fre}), 
\[
\E[f(Y)]=0, \quad\text{ for any $N-$exchangeable
$Y=(Y_1,\ldots,Y_N)$}.
\]
Together with the symmetry of $f$, this implies that $f$ is identically $0$.
\end{proof}

Suppose now that $\|\bm[g\bm]\|=0$. Then $g \sim U^N_k g$ for some $k$
and $N$. By Lemma \ref{claiminproof}, $U^N_k g$ is identically zero.
Thus $\bm[g\bm]=\bm[0\bm]$.
We have thus shown that $\Phieq$ is a normed linear space.
Consider now
\[
\bm[\Phi_n\bm] := \{\bm[f\bm]:\, f \in \Phi_n\}.
\]
Clearly, $\bm[\Phi_n\bm]$ is a linear subspace of $\Phieq$.
It is normed by the same norm.
We now attempt to define a linear functional
\[
\L^0 : \bm[\Phi_n\bm]  \to \R
\]
based on the following observation.
If $f, g \in \Phi_n$ have $\E f(X_1, \ldots, X_n)\neq \E g(X_1,\ldots, X_n)$
Then, by Lemma \ref{pcrucial}, $U^N_n f \neq U^N_n g$
for all $N \ge n$.
So then $f \not \sim g$ and so $\bm[f\bm] \not = \bm[ g\bm]$.
Therefore,
\begin{equation}
\label{L0}
\L^0: \bm[g\bm] \mapsto \E g(X_1, \ldots, X_n)
\end{equation}
is a function; in fact, a linear function from $\bm[\Phi_n\bm]$ 
into $\R$.
Consider the norm of $\L^0$:
\[
\|\L^0\| = \sup_{g \in \Phi_n}
\frac{|\L^0(\bm[g\bm])|}{\|\bm[g\bm]\|}
= \sup_{g \in \Phi_n} \sup_{N \ge n}
\frac{|\E g(X_1, \ldots, X_n)|}{\|U^N_n g\|}
= \sup_{N \ge n} \|\mathcal E^N_n\|,
\]
where the last equality is due to \eqref{Enorm}.
Since, by assumption, $(X_1,\ldots,X_n)$ is $N$-exchangeable
for all $N \ge n$, we have (Lemma \ref{necessary}) that $\|\mathcal E^N_n\|=1$
for all $N \ge n$ and so 
\[
\|\L^0\|=1.
\]
The Hahn-Banach theorem guarantees that there is a linear functional
\[
\LL^* : \bm[\Phi^*\bm] \to \R
\]
such that $\LL^* = \LL^0$ on $[\Phi_n]$ and 
\[
\|\L^*\| = \|\L^0\|=1.
\]
We then define 
\begin{equation}
\label{Ldec}
\L: \Phi^* \to \R; \quad \L g:= \L^*(\bm[g\bm]).
\end{equation}
Note that $\L$ is a linear functional which is moreover symmetric,
that is, $\L g = \L g'$
if $g'$ is obtained from $g$ by permuting its arguments.
Since $\|\L^*\|=1$, we have, for all $g \in \Phi^*$,
\[
|\L g| = |\L^*(\bm[g\bm])| \le \|\bm[g\bm]\| 
= \inf_N \|U^N_{i_g} g \| \le \|g\|,
\]
and so
\[
\|\L\|=1.
\]
In particular, for $N \ge n$, let
\begin{equation}
\label{Lrest}
\L_n^N := \L\big|_{\Phi_N},
\end{equation}
the restriction of $\L$ onto $\Phi_N$.
Then
\begin{equation}
\label{stillone}
\|\L_n^N\|=1.
\end{equation}
We now claim that $\L_n^N$ is an $N$-extending functional, that is,
\begin{equation}
\label{Lsame}
\L_n^N = \mathcal E^N_n \quad\text{on } U^N_n \Phi_n.
\end{equation}
To see this, let $f = U^N_n g$ for
some $g \in \Phi_n$. 
Then $\mathcal E^N_n f = \E g(X_1,\ldots,X_n)$.
On the other hand,
\[
\L_n^N f \stackrel{\text{\eqref{Lrest}}}{=} \L(U^N_n g) 
\stackrel{\text{\eqref{Ldec}}}{=} \L^*(\bm[ U^N_n g\bm])
= \L^*(\bm[ g\bm]) 
= \L^0(\bm[ g\bm]) 
\stackrel{\text{\eqref{L0}}}{=} \E g(X_1,\ldots,X_n).
\]
Note that the symmetry of $\L_n^N$ is inherited from $\L$. 
Therefore $\L_n^N$
is an $N$-extending functional of $(X_1,\ldots,X_n)$. 

Since $\L_n^N$ was constructed via the operator $\L$,
we have the consistency property:
\begin{equation}
\label{cons}
\LL_n^N(\1_A) = \LL_n^{N'}(\1_{A \times S^{N'-N}}), \quad n \le N \le N',
\end{equation}
for all $A$ in the algebra of measurable rectangles of $S^N$.
As in the proof of Theorem \ref{ke} (see Remark \ref{proofnew}), 
we have that there is a probability measure, say $Q_n^N$, on $S^N$,
such that
\[
Q_n^N(A) = \LL_n^N(\1_A),
\]
for all $A$ in the algebra of measurable rectangles of $S^N$.
So \eqref{cons} implies that
\[
Q_n^N(A) = Q_n^{N'}(A \times S^{N'-N}), \quad n \le N \le N',
\]
for all $A$ in the algebra of measurable rectangles of $S^N$.
Moreover, 
\[
\P(X_1,\ldots,X_n \in A) = Q_n^N(A \times S^{N-n}),
\]
for all $A \in \SS^n$ and all $N \ge n$.
By Kolmogorov's extension theorem \cite[p.\ 82]{NevFei65},
there exists a probability measure $Q$ on $(S^{\N}, \SS^{\N})$
such that $Q(A \times S^\infty)= Q_n^N(A)$ if $A \in \SS^N$,
for all $N \ge n$.
By the $N$-exchangeability of $Q_n^N$, for all $N$, we have that
$Q$ is an  exchangeable probability measure
on $(S^\N, \SS^\N)$.  
Let $Y=(Y_1,Y_2,\ldots)$ be a random element of $S^\N$
with law $Q$. Then
$(X_1,\ldots,X_n) \eqdist (Y_1, \ldots, Y_n)$.
This completes the proof.
\end{proof}

\section{A representation result for finite exchangeability}
\label{true}
It is natural to ask under what conditions can the 
representation formula \eqref{fre}
for an $n$-exchangeable probability measure  hold
with $\nu$ a probability measure.
We give a criterion in terms of a functional defined below
that uses the notions and the theorem developed in this paper.

Let $X=(X_1,\ldots,X_n)$ have exchangeable law in $S^n$. 
For $g: S^n \to \R$ bounded and measurable and $\pi$ a
probability measure on $S$ let
\begin{equation}
\label{In}
I(g,\pi) := \int_{S^n} g(x_1,\ldots, x_n) \pi(dx_1) \cdots \pi(dx_n).
\end{equation}
Clearly, $g \mapsto I(g,\pi)$ is linear. Let
\[
\|I(g,\cdot)\| := \sup_{\pi\in \PP(S)} |I(g,\pi)|,
\]
which, for $g$ bounded and measurable, is finite since 
$\|I(g,\cdot)\| \le \|g\|$.
A simple consequence of \eqref{fre}
is that if $I(g,\pi)=I(h,\pi)$ for all $\pi \in \PP(S)$ then
$\E g(X_1,\ldots,X_n) = \E h(X_1,\ldots,X_n)$.

Hence the assignment
\[
\T : I(g,\cdot) \mapsto \E g(X_1,\ldots,X_n)
\]
is a function from the linear space $\{I(g,\cdot),\, g \in b(S^n)\}$
into $\R$. Clearly, it is linear; it is also bounded because
\[
|\E g(X_1,\ldots,X_n)| = \left|\int_{\PP(S)} I(g,\pi)\, \nu(d\pi)\right|
\le \|I(g,\cdot)\|\, \|\nu\|,
\]
where
\[
\|\nu\| = \nu^+(S)+\nu^-(S),
\]
which is finite by Theorem \ref{rrr}.
So, the norm $\|\mathcal T\|$ of $\mathcal T$ satisfies
\begin{equation}
\label{Tnnorm}
1\le \|\T\| = \sup_{g} \frac{|\E g (X_1, \ldots, X_n)|}{\|I(g,\cdot)\|}
\le \|\nu\| < \infty,
\end{equation}
the inequality being true for any signed measure $\nu$ satisfying 
\eqref{fre}.
That $\|\T\| \ge 1$ is clear from the choice $g=1$.
In the sequel, we will assume that $S$ is a locally compact Hausdorff space.
The Baire $\sigma$-algebra of $S$ is the $\sigma$-algebra generated
by the class $C_c(S)$ of continuous functions $f: S\to \R$ with compact 
support.
This, as in Hewitt and Savage \cite{HewSav55},
will guarantee that de Finetti's theorem holds.
Other conditions are, of course possible. For example,
assuming that $S$ is a locally compact Polish space will also work.
\begin{theorem}
\label{infie}
Let $X=(X_1, \ldots, X_n)$ be an exchangeable random element of $S^n$.
Suppose that the hypotheses of Theorem \ref{ke} hold.
In addition, assume that $\SS$ is the Baire $\sigma$-algebra. 
Then the following three assertions are equivalent: 
\begin{enumerate}
\item  $X$ is $N$-extendible for all $N \ge n$ (or, equivalently, 
by Theorem \ref{ine}, $X$ is infinitely extendible)
\item $\|\mathcal{T}\|=1;$
\item there exists a probability measure $\nu$ on $\PP(S)$ 
satisfying (\ref{fre}).
\end{enumerate}
\end{theorem}
\begin{proof}
$1\Rightarrow 3$: If $X=(X_1,\ldots,X_n)$ is $N$-extendible for all $N \ge n$ 
there is, by Theorem \ref{ine}, an exchangeable 
random sequence $Y= (Y_1, Y_2, \ldots)$ such that
$(X_1,\ldots,X_n) \eqdist (Y_1,\ldots, Y_n)$.
Since $\SS$ is the Baire $\sigma$-algebra, 
Theorem 7.4 of Hewitt and Savage \cite{HewSav55} applies:
there exists a probability measure $\nu$ on $\PP(S)$ such that:
\[
\P(Y\in A)=\int_{\PP(S)}\pi^{\infty}(A)\,\nu(d\pi), \quad A\in\SS^{\N},
\]
and hence 
\[
\P(X\in B)=\int_{\PP(S)}\pi^{n}(B)\,\nu(d\pi), \quad B\in\SS^{n}.
\]
$3\Rightarrow 2$:
Since the last display holds with $\nu$ a probability  measure,
\eqref{Tnnorm} gives $\1\le \|\mathcal T\| \le \|\nu\|=1$.
\\
$2\Rightarrow 1$: Let $N\geq n$.
Using symmetry, \eqref{In} gives
\[
I(g,\pi) =\int_{S^{N}}(U^N_n g)(x_1,\ldots,x_N)\,\pi(dx_1)\cdots\pi(dx_N).
\]
Hence
\[
\|I(g,\cdot)\| 
\le \sup_{x\in S^{N}}|(U^N_n g)(x)|=\|U^N_n g\|.
\]
By \eqref{Tnnorm}, and \eqref{Enorm}
\[
1=\|\T\| \ge \sup_{g} \frac{|\E g(X)|}{\|U^N_ng\|}
=\|\mathcal E^N_n\|,
\]
Since $|\mathcal E^N_n\| \ge 1$, we have $\|\mathcal E^N_n\|=1$.
By Theorem \ref{ke} we conclude that
$(X_1,\ldots,X_n)$ is $N$-extendible.
\end{proof}

\section{Additional results and comments}
\label{addi}

\subsection{A different version of Theorem \ref{ke}.}
Theorem \ref{ke} proved in this paper states that, under suitable
assumptions, an exchangeable random element $(X_1,\ldots, X_n)$ of $S^n$
is $N$-extendible if and only if
$\sup_g |\E g(X)|/\|U^N_n g\|=1$, where the supremum is taken over
all bounded measurable functions $g: S^n \to \R$. We wish to 
see whether the supremum can be reduced to a smaller class of functions.
For example, suppose that $\SS_0$ is an algebra generating the Borel
sets $\SS$ of $S$. Replacing the tightness and outer regularity assumptions
by assumptions that involve the class $\SS_0$ allows us to consider the
supremum over the class of sets that are linear combinations of 
of indicators $\1_{V_1\times\cdots\times V_n}$ with $V_i \in \SS_0$.
The assumptions imposed by the next theorem are satisfied in all
natural cases.
The proof follows closely the proof of Theorem \ref{ke} and thus
is only sketched.
\begin{theorem}
\label{keext}
Let $\SS_0$ be an algebra of subsets of the locally compact Hausdorff
space $S$ generating its Borel sets. Let $(X_1,\ldots,X_n)$ be 
an exchangeable random element of $S^n$. Assume:
\\
(a) For all $\epsilon >0$ there exists $V \in \SS_0$ and compact set $K$
such that $V \subset K$ and  $\P(X_1 \in V) \ge 1-\epsilon$.
\\
(b) For all $\epsilon >0$ and all $V \in \SS_0$ there exists open
set $O$ and $W \in \SS_0$ such that $V \subset O \subset W$ and
$\P(X_1 \in W\setminus V) \le \epsilon$.
\\
Then, for $N > n$, $(X_1,\ldots,X_n)$ is $N$-extendible if and only if
\begin{equation}
\label{gres}
\sup_g |\E g(X)|/\|U^N_n g\|=1,
\end{equation}
where $g$ in the supremum ranges over
the set of all 
linear combinations
of indicators $\1_{V_1\times\cdots\times V_n}$ with $V_i \in \SS_0$.
\end{theorem}
\begin{proof}[Sketch of proof]
We only need to prove sufficiency. 
Denote by $\mathfrak F$ the class of functions that are linear combinations 
of indicators $\1_{V_1\times\cdots\times V_n}$ with $V_i \in \SS_0$.
Assume that the 
supremum in \eqref{gres}, over the class $\FF$ of $g$'s equals 1.
This assumption is equivalent to having the norm of $\mathcal E^N_n$ on
$\mathfrak F$ equal to 1. As in Lemma \ref{propertiesofL}(i),
we can extend this restriction to the set $b(S^N)$ of all
bounded measurable functions on $S^N$, using the Hahn-Banach theorem,
without increasing its norm. We can also assume that the extension is
a symmetric operator. Denote it by $\LL$. Since $\LL(1)=1$ and 
$\|\LL\|=1$ we have, as in Lemma \ref{propertiesofL}(ii)
that $\LL$ is monotone and hence $\LL(\1_A) \ge 0$ for all
Borel $A \subset S^N$.
Using (a), and arguing precisely as in Lemma \ref{Lboundedcompacts},
we obtain that $\LL$ has norm $1$ on $C_c(S^N)$.
We then proceed as in the proof of Theorem \ref{ke} to
extract, via Riesz representation, a probability measure $Q$ on $S^N$
that satisfies \eqref{Ropen} and which is also symmetric. 
It then suffices to show that if $R=V_1\times \cdots \times V_N$
is a rectangle with $V_i \in \SS_0$, we have
$Q(R) \le \LL(\1_R)$.
Let $\epsilon > 0$.
Using (b) we select $W_{i,\epsilon} \in \SS_0$ 
and open sets $O_{i,\epsilon}$ such
that $V_i \subset O_{i,\epsilon} \subset W_{i,\epsilon}$
and $\P(W_{i,\epsilon} \setminus V_i) \le \epsilon$.
This is used to prove
that $G_\epsilon:= (W_{1,\epsilon} \times \cdots \times W_{N,\epsilon})
\setminus R$
satisfies
\[
\LL(\1_{G_\epsilon}) \le N^2 \epsilon.
\]
Since $O_\epsilon:= O_{1,\epsilon} \times \cdots \times O_{N,\epsilon}$
is an open set containing $R$ we have
\[
Q(R) \le Q(O_\epsilon) \le \LL(\1_{O_\epsilon}) \le
\LL(\1_R)+N^2\epsilon,
\]
where we used \eqref{Ropen} for the second inequality
and the previous display for the last. Letting 
$\epsilon \downarrow 0$, we conclude.
\end{proof}
\begin{remark}
Property (a) of Theorem \ref{keext} clearly implies tightness.
Using the monotone class theorem, we can also show that 
property (b) implies outer regularity.
On the other hand, the assumptions of
Theorem \ref{ke} do imply those of Theorem \ref{keext} provided that
the algebra $\SS_0$ and the space $S$ are suitable;
for example, with $S = \R$ and $\SS_0$ the algebra generated
by intervals.
\end{remark}

\subsection{Extendibility under limits.}
First, we observe that extendibility is a property that remains
true under limits in total variation.
Assume that $S$ is a locally compact Hausdorff space.
Let $X^i=(X^i_1,\ldots, X^i_n)$, $i=1,2,\ldots$, be a sequence
of exchangeable random elements that are $N$-extendible such that
$X^i$ converges to $X=(X_1,\dots,X_n)$ in total variation.
Assume that the law of $X_1$ is tight and outer regular.
Then $X$ is $N$-extendible.
Indeed, $X$ is clearly exchangeable, but that it is $N$-extendible:
By the total variation convergence we have
\[
\lim_{i \to \infty} \E g(X^i) = \E g(X)
\]
for all bounded measurable $g:S^n \to \R$. If $\mathcal E^N_n$
is the primitive extending functional of $X$ then, by
\eqref{Enorm}, we have $\|\mathcal E^N_n\| \le 1$.
The claim follows from Theorem \ref{ke}.

On the other hand, we have the following result that, roughly speaking,
says that if we have extendible probability measures $P_k$ on coarse 
$\sigma$-algebras whose union generates the full $\sigma$-algebra in
a way that $P_k$ converges to $P$ in a certain sense, then
$P$ is extendible:
\begin{theorem}
\label{extcoarse}
Let $\SS_0$ be an algebra of subsets of $S$ generating $\SS$.
Let all assumptions of Theorem \ref{keext} hold.
In addition, suppose that there is 
an increasing family $\GG_1 \subset \GG_2 \subset \cdots$ of 
$\sigma$-algebras on $S$ such that $\bigcup_k \GG_k=\SS_0$.
For each $k$, let $P_k$ be a probability measure on $S^n$ defined
on $\GG_k^{n}$ (the product $\sigma$-algebra) such that 
\begin{equation}
\label{coarsesense}
\lim_{k \to \infty} \sup_{A \in \GG_k^n} |P_k(A)-P(A)| =0,
\end{equation}
where $P$ is the law of $(X_1,\ldots, X_n)$ and assume
that, for some $N > n$, and all $k \ge 1$, $P_k$ is $N$-extendible.
Then $P$ is $N$-extendible.
\end{theorem}
\begin{proof}
Let $g \in \mathfrak F$, the class of real-valued functions on $S^n$
that are linear combination of indicators $\1_{V_1\times \cdots \times V_n}$
with $V_i \in \SS_0$ for all $i$. 
Then there is $k$ such that $g$ is $\GG_k$-measurable.
Our assumption then implies that $\int_{S^n} g dP_k \to \int_{S^n} g dP$
as $k \to \infty$. We now appeal to Theorem \ref{keext}.
Since $P_k$ is $N$-extendible, we have
$\sup_{g \in \mathfrak F} \big|\int g dP_k \big|/\|U^N_n g\|=1$.
Hence $\sup_{g \in \mathfrak F} \big|\int g dP \big|/\|U^N_n g\|=1$
also and so, by Theorem \ref{keext}, $P$ is $N$-extendible.
\end{proof}

We give an example of how this can be applied using  a known example
of an extendible distribution. Gnedin \cite{Gnedin96} shows that
if $(X_1,\ldots,X_n)$ is a random element of $\R^n_+$ with
density $f(x_1,\ldots,x_n)=g(x_1 \vee \cdots \vee x_n)$ for
some decreasing $g$ then it is infinitely extendible.
We will provide an alternative proof for this by constructing,
for each $j \in \N$, an infinitely-extendible probability
measure $P_j$ on some coarse $\sigma$-algebra $\FF_j$ of $S^n$,
increasing with $j$, such that $P_j$ approaches the law of
$(X_1,\ldots,X_n)$ in the sense of \eqref{coarsesense}.
Without loss of generality, let $n=2$.

Fix $j \in \N$. Let $\D_j$ be the set of rational numbers
$k/2^j$ for $k=1,\ldots, j 2^j$. This splits the positive
real line into a finite number of intervals: the bounded intervals
$I_k(j) := [(k-1)2^{-j}, k2^{-j})$ and the interval $I_0(j):=[j,\infty)$.
We let $\GG_j = \sigma(\Pi_j)$ and $\SS_0 = \bigcup_j \GG_j$.
It is easy to see that (a) and (b) of Theorem \ref{keext}
hold for the algebra $\SS_0$.
Next, let $\FF_j = \GG_j \otimes \GG_j$,
the corresponding product $\sigma$-algebra on $[0,\infty)\times[0,\infty)$.
To specify probability measure $P_j$ on the sets of $\FF_j$ it is enough
to specify it on the sets $I_k(j) \times I_\ell(j)$.
We let
\[
P_j(I_k(j) \times I_\ell(j)) = c_j\, g((k\vee \ell)/2^j), 
\quad 1\le k,\ell \le j2^j,
\]
and set $P_j(I_k(j) \times [j,\infty)) = P_j( [j,\infty) \times I_k(j))=0$.
The constant $c_j$ is just a normalization constant.
By construction, $P_j$ is exchangeable. We see that it
is infinitely extendible by observing that it is a mixture
of product measures.

For $1 \le r \le j2^j$, define the product probability measure $Q_r$
on $\FF_j$ by
\[
Q_r(I_k(j) \times I_\ell(j)) = \frac{1}{r^2},
\quad 1\le k,\ell \le j2^j,
\]
while $Q_r(I_k(j) \times [j,\infty)) = Q_r( [j,\infty) \times I_k(j))=0$.
Then, with $a_r = r/2^j$ for $r \le j 2^j$ and $a_{j2^j+1}=0$,
\[
P_j = c_j \sum_{r=1}^{j 2^j} [g(a_r)-g(a_{r+1})] r^2 Q_r  ,
\]
and the coefficients are positive due to the monotonicity of $g$.

We finally observe that, for all $i \in \N$,
\[
\max_{1 \le k,\ell \le i 2^i}
\bigg| P_j(I_k(i) \times I_\ell(i)) - \int_{I_k(i) \times I_\ell(i)}
g(x \vee y)\, dx dy \bigg| \to 0, \quad \text{as  $j \to \infty$},
\]
meaning that condition \eqref{coarsesense} of Theorem \ref{extcoarse}
is verified.

\section*{Appendix A.}

\paragraph{\bf 1) Covariance.}
Exchangeability imposes strong conditions on covariance for
second-order random variables. That is, if $(X_1,\ldots,X_n)$
is $n$-exchangeable random element of $\R^n$ with $\E X_1^2<\infty$
then it is easy to see  that $\cov(X_1,X_2) \ge -\var(X_1)/(n-1)$.
On the other hand, if $(X_1,X_2,\ldots)$ is an exchangeable sequence
of real random variables with finite variance then $\cov(X_1,X_2) \ge 0$.
However, if $(X_1,\ldots, X_n)$ is $n$-exchangeable, nonnegativity
of $\cov(X_1,X_2)$ is not at all sufficient for infinite extendibility.
For example, take $n=2$, $S=\{s_1,s_2,s_3,s_4\} = \{1,\,1.5,\,2,\,2.5\}$
and let $(X_1,X_2)$ take values $(s_1,s_2)$, $(s_2, s_1)$, $(s_3,s_4)$,
$(s_4,s_3)$ with probability $1/4$ each.
Then $\cov(X_1,X_2)= 3/16$ but it is easy to see that
\eqref{fre} cannot hold with $\nu$ a probability measure.
By Theorem \ref{infie} $(X_1,X_2)$ is not infinitely extendible.

\paragraph{\bf 2) An example of an extending functional not defining
a probability measure.}
We give an example of an extending functional $\LL$ such that
$A \mapsto \LL(\1_A)$ is not a probability measure.
See Remark \ref{Lnotmeas}.
Let $S=[0,1]$, the closed unit interval. Take $n=1$ and $N=2$
and start with the probability measure on $[0,1]$ to be the
uniform measure on the Borel sets $\BB$. 
Let $\Phi_i$ be the set of bounded measurable real-valued
functions on $[0,1]^i$, $i=1,2$.

For $g \in \Phi_1$ we have $(U_1^2 g)(x_1,x_2)= \frac{1}{2}(g(x_1)+g(x_2))$.
The primitive extending functional $\mathcal E$ maps $U_1^2 g$
to $\int_0^1 g(t) dt$. 

We now construct a particular $2$-extending functional $\L$.
We let $\mathfrak F$ consist of functions of the form
\begin{equation}
\label{Fform}
F(x,y)= \sum_{i=0}^m c_i \1\{x \in A_i,y\in B_i\} , \quad x_i \in \R,
\, A_i, B_i \in \BB.
\end{equation}
$D:=\{(t,t):\, 0 \le t \le 1\}$ and
let $\mathfrak D$ consist of functions of the form 
\begin{equation*}
G=F+c1_D, \quad F \in \mathfrak F, \quad c \in \R.
\end{equation*}
Define the symmetric functional $\L_0: \mathfrak D \to \R$ by
\begin{equation}
\label{LD}
\L_0(F+c1_D) := \E F(X,X) = \int_0^1 F(t,t)\, dt.
\end{equation}
We claim that $\|\L_0\|=1$. See below for the proof of this claim.
Since $\mathfrak D$ is a normed linear subspace of $\Phi_2$
there exists (by the Hahn-Banach theorem) a symmetric linear functional
$\L: \Phi_2 \to \R$  such that
$\L=\L_0$ on $\mathfrak D$ and $\|\L\| = \|\L_0\|=1$.

We show that $\L$ is a 2-extending functional, i.e.,
that $\L(U_1^2 g) = \int_0^1 g(t)\, dt$, for all $g \in \Phi_1$.
Let $f_n$ (respectively, $h_n$) be an increasing
(respectively, decreasing) sequence of simple functions on $[0,1]$
(i.e., linear combinations of finitely many 
indicator functions of Borel subsets of $[0,1]$)
such that $f_n \uparrow g$ (respectively, $h_n \downarrow g$).
We have $f_n \le g \le h_n$ for all $n$, and so 
$U_1^2f_n \le U_1^2g \le U_1^2h_n$.
Since $\|\L\|=1$,
by Lemma \ref{Lsimple}, 
$\L$ is a monotone operator.
Hence $\L(U_1^2f_n) \le \L(U_1^2g) \le \L(U_1^2h_n)$
for all $n$. Since $U_1^2 f_n \in \mathfrak F$, we have
$\L(U_1^2 f_n) = \L_0(U_1^2 f_n) = \int_0^1 f_n(t)\, dt$.
Similarly, $\L(U_1^2 h_n) = \int_0^1 h_n(t)\, dt$.
By monotone convergence, $\lim_{n\to\infty}\int_0^1 f_n(t)\, dt =
\lim_{n\to\infty}\int_0^1 h_n(t)\, dt = \int_0^1 g(t)\, dt$.
Therefore, $\L(U_1^2g) = \int_0^1 g(t)\, dt = \mathcal E (U_1^2 g)$,
showing that $\L$ is an extension of $\mathcal E$. 

As in the proof of Theorem \ref{ke}, the functional $\L$ 
restricted on the space $C([0,1]\times[0,1])$ of continuous functions on $[0,1]
\times [0,1]$
admits the Riesz representation 
\[
\L F = \int_{[0,1]\times[0,1]} F(x,y)\, Q(dx,dy),
\quad F \in C([0,1]\times[0,1]),
\]
for some probability measure $Q$
on $[0,1]\times[0,1]$ and this $Q$ is a $2$-extension of the
law of $X$. 
It is easy to see that $Q$ is the law of $(X,X)$.

We have (as in the proof of Theorem \ref{ke}) 
$\LL(\1_R) = Q(R)$ for all rectangles $R=A\times B$, $A,B \in \BB$.

If $A \mapsto \L(\1_A)$ were a probability measure
on the Borel sets $A$ of $[0,1]\times [0,1]$ we would certainly have
$\L(\1_A) = Q(A)$ for all Borel $A \subset [0,1]\times [0,1]$.
But $\L(\1_D) =\L_0(\1_D)=0$ in contradiction to
$Q(D) = \P((X,X)\in D) = 1$.

\begin{proof}[Proof of the claim that $\L_0$ has norm $1$]
We need to show that $|\L_0(F + c \1_D)| \le \|F+c \1_D\|$, for
all $F \in \mathfrak F$ and all $c \in \R$. If $c=0$, the inequality holds.
If $c\neq 0$, divide by $c$ and use \eqref{LD} to reduce the claim
to the proof of the inequality
\[
\bigg|\int_0^1 F(t,t)\, dt \bigg| 
\le  \max_{x\neq y} |F(x,y)| \vee \max_t |F(t,t)+1|,
\quad F \in \mathfrak F.
\]
We consider two cases. If $\max_{x \neq y} |F(x,y)| =
\max_{x,y} |F(x,y)|$ then
$$\big|\int_0^1 F(t,t)\, dt \big|  \le \max_{x,y} |F(x,y)|
= \max_{x \neq y} |F(x,y)| \le \max_{x\neq y} |F(x,y)| \vee \max_t |F(t,t)+1|.$$
If not, there is $t_0 \in [0,1]$ such that $|F(t_0,t_0)| > \max_{x \neq y}
|F(x,y)|$.
Since $F$ can be written as in \eqref{Fform} with pairwise disjoint
$R_i$, it follows that one of the $R_i$ must be a singleton,
say, $R_0=\{t_0\}\times\{t_0\}$. Without loss of generality,
assume that this is the only singleton among the $R_i$'s.
Then $F = c_0\1_{R_0} + H$, where $H = \sum_{i=1}^m c_i \1_{R_i}$
has the property of case 1, i.e., 
$\max_{x \neq y} |H(x,y)| = \max_{x,y} |H(x,y)|$.
Then 
\begin{align*}
\big|\int_0^1 F(t,t)\, dt \big|
= \big|\int_0^1 H(t,t)\, dt \big|
&\le \max_{x,y}|H(x,y)|
= \max_{x \neq y} |H(x,y)| 
= \max_{x \neq y} |F(x,y)| &
\\
&\le \max_{x \neq y} |F(x,y)| \vee \max_t |F(t,t)+1|.&
\end{align*}
\end{proof}







\end{document}